\documentclass{amsart}

\usepackage[all]{xy}
\usepackage{hyperref}
\hypersetup{
    colorlinks=true,       
    linkcolor=blue,          
    citecolor=blue,        
    filecolor=blue,      
    urlcolor=blue           
}
\usepackage{tikz}
\usetikzlibrary{graphs,positioning,arrows,shapes.misc,decorations.pathmorphing}

\tikzset{
    >=stealth,
    every picture/.style={thick},
    graphs/every graph/.style={empty nodes},
}

\tikzstyle{vertex}=[
    draw,
    circle,
    fill=black,
    inner sep=1pt,
    minimum width=5pt,
]
\usepackage[position=top]{subfig}
\usepackage[backrefs,lite]{amsrefs}
\usepackage{ amssymb}
\usepackage{todonotes}
\usepackage{stmaryrd}
\usepackage{tikz-cd}

\newcommand{\supp}{\operatorname{supp}}

\newcommand{\Ch}{\operatorname{Ch}}

\newcommand{\qq}{\mathbb{Q}}
\newcommand{\zz}{\mathbb{Z}}

\newcommand{\rr}{\mathbb{R}}
\newcommand{\cc}{\mathbb{C}}

\newtheorem{introthm}{Theorem}

\newtheorem{theorem}{Theorem}[section]
\newtheorem{lemma}[theorem]{Lemma}

\newtheorem{claim}[theorem]{Claim}

\DeclareRobustCommand{\rchi}{{\mathpalette\irchi\relax}}
\newcommand{\irchi}[2]{\raisebox{\depth}{$#1\chi$}} 

\theoremstyle{definition}

\newtheorem{notation}[theorem]{Notation}
\newtheorem{definition}[theorem]{Definition}
\newtheorem{example}[theorem]{Example}

\newtheorem{remark}[theorem]{Remark}

\theoremstyle{remark}

\numberwithin{equation}{section}

\begin{document}

\title{Bounding Singular Surfaces Via Chern Numbers}

\author[J.~Moraga]{Joaqu\'in Moraga}
\address{
Department of Mathematics, University of Utah, 155 S 1400 E, 
Salt Lake City, UT 84112}
\email{moraga@math.utah.edu}
\thanks{ The author was partially supported by NSF research grants no: DMS-1300750, DMS-1265285 
and by a grant from the Simons Foundation; Award Number: 256202}

\maketitle

\medskip 

\begin{abstract}
We prove the existence of a bound on the number
of steps of the minimal model program for singular surfaces 
in terms of discrepancies and top Chern numbers.
As an application, we prove that given $R\in \mathbb{R}_{>0}$ and $\epsilon \in (0,1)$, 
the class $\mathcal{F}(R,\epsilon)$ of $2$-dimensional pairs $(X,D)$ of general type
with $\epsilon$-klt singularities, $D$ with standard coefficients, and $4c_2(X,D)-c_1^2(X,D)\leq R$, forms a bounded family.
\end{abstract}

\setcounter{tocdepth}{1}
\tableofcontents

\section*{Introduction}

The aim of this article is to prove that a class
of complex algebraic singular log surfaces of general type 
forms a bounded family whenever the Chern numbers
and the discrepancies are bounded. 
In order to prove such statement we study 
a topological invariant for singular surfaces which strictly
decreases under divisorial contractions and we prove that
this invariant is discrete and non-negative in such sets.
We will use techniques of the minimal model program, abbreviated as MMP, and generalizations of the
Bogomolov-Miyaoka-Yau inequality, or BMY for short.

It is known that given a smooth surface,
after finitely many contractions of $(-1)$-curves, 
we arive to either a ruled surface or a surface with semiample canonical divisor.
One of the main purposes of the minimal model program is to generalize this picture
for possibly singular higher dimensional pairs. 
Many important steps of the MMP have been achieved, 
for example for $2$-dimensional log pairs 
there are explicit classifications of the singularities of the MMP (see ~\cite{Flips} and ~\cite[Chapter 4]{KM98})
and we have the boundedness of minimal models by Alexeev (see ~\cite{Ale94} and ~\cite{AM00}).
However, not much is known about which invariants of an algebraic variety 
can bound the numbers of steps of its MMP.
In this article we intend to prove a result in this direction for singular surfaces.

One of the main ingredients that we use is the Bogomolov-Miyaoka-Yau inequality: 
Bogomolov proved in ~\cite[Theorem 5]{Bog78} that for a smooth projective surface $X$
the inequality $4c_2(X)\geq c_1(X)^2$ holds,
then Miyaoka ~\cite[Theorem 4]{Miy77} and Yau ~\cite[Theorem 4]{Yau77} improved such inequality
to $3c_2(X)\geq c_1(X)^2$. 
It is known that this latter inequality is sharp
and for surfaces of general type the equality holds
if and only if $X$ is isomorphic to the quotient of the two dimensional ball
by an infinite discrete group.
Some further generalizations for singular surfaces with boundary with standard coefficients
were obtained, for example, 
by Sakai~\cite[Theorem 7.6]{Sak80}, Miyaoka~\cite[Theorem 1.1]{Miy84} 
and Megyesi~\cite[Theorem 0.1]{Meg99}.
Finally, Langer proved in ~\cite[Corollary 0.2]{Lan01} 
a more general inequality for log canonical surfaces over the complex numbers. 

Now we turn to state the precise result of this article.
In what follows we consider normal projective varieties over the field of complex numbers $\cc$.
Given two real numbers $R\in \rr_{>0}$ and $\epsilon \in (0,1)$ we denote by
$\mathcal{F}(R,\epsilon)$ the class
of $2$-dimensional pairs $(X,D)$ of general type with
$\epsilon$-klt singularities, 
such that $D$ has {\em standard coefficients},
and $4c_2(X,D)-c_1(X,D)^2 \leq R$. 

\begin{introthm}\label{maintheorem}
Let $ R \in \rr_{>0}$ and $\epsilon \in (0,1)$.
Then the class $\mathcal{F}(R,\epsilon)$
forms a bounded family.
\end{introthm}

See ~\ref{standard} for the definition of standard coefficients, 
and ~\ref{bounded} for the definition of bounded family. 
The idea of the proof is to 
investigate a topological invariant introduced by Megyesi in ~\cite{Meg99} and study how it changes for divisorial contractions of $\epsilon$-klt pairs.
We prove that for elements in $\mathcal{F}(R,\epsilon)$ such invariant takes values on a finite subset of $\mathbb{R}_{>0}$
which only depends on $R$ and $\epsilon$.
As a consequence we deduce the existence of a
bound, which only depends on $R$ and $\epsilon$, for the number of steps of the minimal model program of a member of  $\mathcal{F}(R,\epsilon)$.
Using this latter fact we will conclude that the class $\mathcal{F}(R,\epsilon)$ forms a bounded family.

\subsection*{Acknowledgements}
The author would like to thank Christopher Hacon and
Valery Alexeev for many useful comments.

\section{Preliminaries}

In this section we recall some usual definitions
from birational geometry and Chern classes, 
and we state some preliminary results
that will be used in the proof of the main theorem.
We often use the standard notation of
~\cite{Har77}, ~\cite{KM98} and ~\cite{Koll13}. 

\begin{definition}\label{boundedschemes}
A {\em pair of schemes} is a couple $(X,Z)$ where $Z$ is a sub-scheme of $X$. 
A class $\mathcal{C}$ of pair of schemes $\{(X,Z)\}$ is said to be a {\em bounded family}
if there exists three Noetherian schemes of finite type over the complex numbers $\rchi,\mathcal{Z}$ and $T$,
with $\mathcal{Z}$ is a subscheme of $\rchi$, 
and a morphism of schemes $\phi \colon \rchi \rightarrow T$,
such that for any pair $(X,Z)\in \mathcal{C}$, there exists a closed point $t\in T$, and an isomorphism
$\rchi_t \simeq X$ which induces an isomorphism $\mathcal{Z}_t \simeq Z$.
\end{definition}

\begin{definition}\label{standard}
A {\em $2$-dimensional pair} is a couple $(X,D)$ where
$X$ is a normal projective surface, 
$D$ is a $\qq$-divisor whose coefficients belong to $[0,1]$ 
and $K_X+D$ is a $\qq$-Cartier divisor.
In this article we deal with pairs with {\em standard coefficients}
meaning that the coefficients of $D$ belong to the set
\[ \left\{ 1-\frac{1}{m} \mid m\in \zz_{\geq 1}\cup \{\infty\} \right\}.\]
We say that a pair $(X,D)$ is of {\em general type} if $K_X+D$ is a big $\qq$-divisor.
A {\em resolution of singularities} of a pair $(X,D)$ is a proper
birational morphism $f\colon Y \rightarrow X$ such that $Y$ is a smooth surface,
we say that the resolution is a {\em log resolution} if 
$f^{-1}_*D_{\rm red} \cup E$ has simple normal crossing singularities,
where $E$ is the exceptional locus of $f$ with reduced scheme structure
and $D_{\rm red}$ stands for the reduced divisor supported on $\supp(D)$.
We will say that a resolution $f$ of $X$ is {\em minimal} if any other resolution 
of $X$ dominates $f$.
It is known that any $2$-dimensional pair has a minimal resolution of singularities.
\end{definition}

\begin{definition}\label{bounded}
We say that a class $\mathcal{C}$ of $2$-dimensional log pairs $\{ (X,D)\}$ is {\em bounded}
(or {\em forms a bounded family}), if the class of pairs of schemes $\{ (X,D_{\rm red})\}$ is bounded in the sense of Definition~\ref{boundedschemes},
and the coefficients of the boundaries $D$ belong to a finite set.
As usual, $D_{\rm red}$ is the divisor $D$ with reduced scheme structure.
\end{definition}

\begin{definition}\label{qpolarization}
Given a class $\mathcal{C}$ of $2$-dimensional log pairs $\{ (X,D)\}$,
a {\em $\qq$-polarization} on $\mathcal{C}$ is a class of ample $\qq$-Cartier $\qq$-divisors
$A_X$ on $X$ for each $(X,D) \in \mathcal{C}$.
We say that a {\em $\qq$-polarization} $\mathcal{A}$ is a {\em bounding $\qq$-polarization}
if there exists positive real numbers $C$ and $C'$, and $N\in \mathbb{Z}_{\geq 1}$, such that the following conditions hold
for every $(X,D)\in \mathcal{C}$:
\begin{itemize}
\item $NA_X$ is a Cartier Weil divisor, 
\item $A_X^2 \leq C$, and 
\item $A_X\cdot D \leq C'$.
\end{itemize}
By~\cite[Lemma 3.7]{Ale94}, we know that a class of $2$-dimensional log pairs with a bounding $\qq$-polarization
forms a bounded family in the sense of Definition~\ref{bounded}.
\end{definition}

\begin{definition}\label{def1}
Let $(X,D)$ be a $2$-dimensional pair 
and let $x\in X$.
Consider $f\colon Y \rightarrow X$ to be a log resolution of $(X,D)$, 
so we can write 
\begin{equation}\label{eq}
K_Y  = f^*(K_X+D) + \sum_{i\in I} a_i E_i,
\end{equation}
where $\{ E_i \mid i\in I\}$ is a finite family of distinct divisors,
we denote by $I_x\subseteq I$ the set of divisors whose image in $X$ pass through $x$
and by $I'_x$ the subset of $I_x$ consisting of $f$-exceptional divisors.
We call the $a_i$'s the {\em discrepancies} of $x$ with respect to $f$.
A point $x\in X$ is said to be 
\begin{itemize}
\item {\em $\epsilon$-kawamata log terminal} or {\em $\epsilon$-klt} if $a_i> \epsilon-1$, for every $i\in I_x$.
\item {\em $\epsilon$-purely log terminal} or {\em $\epsilon$-plt} if $a_i>\epsilon-1$, for every $i\in I'_x$.
\end{itemize}
If $\epsilon=0$, then we just omit $\epsilon$ from the notation.
We say that $(X,D)$ is $\epsilon$-klt (resp. $\epsilon$-plt) if all its points
are $\epsilon$-klt (resp. $\epsilon$-plt).
The divisor $D$ is often called the {\em boundary} of the pair.
\end{definition}

\begin{remark}
Kawamata log terminal singularities of dimension $2$, with trivial boundary, are quotient singularities (see, e.g.,~\cite{Tsu83}). 
\end{remark}

\begin{notation}\label{blowupgraphs}
For any klt surface singularity $x$ in $(X,D)$ 
we consider $f\colon Y \rightarrow X$ to 
be the minimal resolution of $x$, 
$E_x$ to be the exceptional locus with reduced scheme structure 
and $D_x$ to be the union of the local analytic branches of $D$
passing through $x\in X$.
We construct a weighted graph $G(X,D;x)$ as follows:
The vertices of $G(X,D;x)$ correspond to the irreducible components
of $f^{-1}_*D_{x} \cup E_x$, 
we do not associate weights to the curves in $f^{-1}_*D_{x}$,
the weight of a curve in $E_x$ is its negative self-intersection
and two vertices will be joined by and edge if such curves intersect.
As usual, the weight of a vertex $v$ will be denoted by $w(v)$.
We just write $G(x)$ instead of $G(X,D;x)$, when the pair $(X,D)$ is clear from the context.

In the proof of the main theorem we will blow up smooth centers
in the minimal resolution of  klt singular points,
to each blow up we can associate a new weighted graph as follows: 
\begin{itemize}
\item 
Blowing up a vertex of $G(x)$:\\ 
Given a vertex $v\in G(x)$, the graph obtained by {\em blowing up $v$}
is $G(x)$ enlarged with a new vertex $v_0$ of weight one which is joined to $v$,
and $w(v)$ is increased by one whenever $v$ is a weighted vertex. 
This new graph represents the non-minimal resolution of $x$ in $(X,D)$
obtained by blowing up a center which is contained in the curve corresponding to
$v$ and is not contained in any other irreducible curve of $f^{-1}_*D_x \cup E_x$.
\item 
Blowing up an edge of $G(x)$: \\
Given an edge $e\in G(x)$, the graph obtained by {\em blowing up $e$}
is $G(x)$ enlarged with a new vertex $v_0$ of weight one
which is joined to the ends of $e$,
the weights of the ends of $e$ are increased by one whenever such vertices are weighted
and $e$ is removed.
Observe that the graph representing the non-minimal resolution of $x$ in $(X,D)$
obtained by blowing up the smooth point corresponding to the edge $e$ 
may differ from this graph at most by some edges from the vertices of $D_x$.
\end{itemize}
\end{notation}

In what follows we turn to recall the definition 
of Chern classes
of a klt pair $(X,D)$ used
in ~\cite{Lan00},~\cite{Lan01} and ~\cite{Lan16}.

\begin{definition}
Given a $2$-dimensional pair $(X,D)$ we define
the first Chern class to be
\[
c_1(X,D) = c_1( \mathcal{O}_X(K_X+D) ),
\]
and $c_1^2(X,D)$ will denote the self-intersection of $K_X+D$.
We consider a finite morphism $f\colon Y\rightarrow X$
such that $f^*D$ is a Weil divisor,
and then we define the second Chern class to be
\[
c_2(X,D) = c_2(Y,\widehat{ f^*\Omega_X }(\log(D)))/\deg(f),
\]
where the wide hat stands for the reflexivization of the sheaf
$ f^*\Omega_X (\log(D))$.
See ~\cite[Definition-Proposition 2.9]{Wahl93}  for
the definition of Chern classes of rank $2$ reflexive sheaves
on normal surfaces.
Whenever we work over pairs with trivial boundary we
will just write $c_1^2(X)$ and $c_2(X)$ 
to denote the usual Chern classes.
\end{definition}

\begin{remark}
In ~\cite[Definition 1.8]{Meg99}
there is a definition of $c_2(X,D)$ as an orbifold Euler number,
using such definition we can compute $c_2(X)$ in terms of
the order of the {\em local fundamental group} of the singularities
of $X$. Indeed, we can write
\[
c_2(X) = \chi_{\rm top}(X) - \sum_{x\in X_{\rm sing}} \left( 1- \frac{1}{r(x)}\right)
\]
where $r(x)$ is the order of the local fundamental group of $X$ at $x$.
In such formula we let the summand on the right to be one whenever the local 
fundamental group is infinite, which happens if and only if $x\in X$ is not klt.
\end{remark}

\begin{example}
Consider $X_{n}$ to be the normalized blow up
of $\cc^2$ defined by the ideal $\langle y,x^n \rangle$. 
This surface has a unique singular $A_n$ point
whose local fundamental group has order equal to $n$
and $\chi_{\rm top}(X_n)=2$, so we obtain
\[ c_2(X_n) = 1+\frac{1}{n}. \]
\end{example}

\begin{definition}
A birational contraction $f\colon (X,D) \rightarrow (X',D')$ of
$2$-dimensional klt pairs, is said to be {\em $(K_X+D)$-negative} if
the $\qq$-Cartier $\qq$-divisor
$-(K_X+D)$ is relatively nef over $X'$. 
\end{definition}

\begin{definition}
Given a $2$-dimensional klt pair $(X,D)$
we say that it has a minimal model $(X_{\rm min},D_{\rm min})$
if there exists a birational contraction $f\colon X \rightarrow X_{\rm min}$
such that $f_*D=D_{\rm min}$, the $\qq$-divisor $K_{X_{\rm min}}+D_{\rm min}$ is
$\qq$-Cartier and nef, and $f$ is $(K_X+D)$-negative.
\end{definition} 

The following theorem is well-known, see for example ~\cite[Theorem 3.48]{KM98}.

\begin{theorem}\label{existence}
Let $(X,D)$ be a $2$-dimensional $\epsilon$-klt pair 
such that $K_X+D$ is pseudo-effective.
Then the minimal model $(X_{\rm min},D_{\rm min})$ of $(X,D)$ exists, 
and has $\epsilon$-klt singularities.
The birational morphism $f\colon X \rightarrow X_{\rm min}$ 
can be factoreds into (finitely many) divisorial contractions
of $\epsilon$-klt pairs.
Moreover, any $(K_X+D)$-negative birational contraction factors $f$.
\end{theorem}

The following theorem is a version of ~\cite[Corollary 0.2]{Lan01}
or ~\cite[Theorem 0.1]{Meg99}
that we will use in this article. 

\begin{theorem}\label{langer}
Let $(X,D)$ be a $2$-dimensional klt pair 
such that $K_X+D$ is pseudo-effective 
and let $(X_{\rm min},D_{\rm min})$ be its minimal model,
then the following inequality holds:
\[
4c_2(X,D) - c_1^2(X,D) \geq \frac{1}{3}c_1(X_{\rm min},D_{\rm min})^2.
\]
\end{theorem}

Now we turn to recall the classification
of klt surface singularities.

\begin{definition}
In this definition given a singular point $x$ in $(X,D)$
we will denote by $f\colon Y\rightarrow X$ its minimal resolution 
with reduced exceptional locus $E_x$,
and $D_x$ the sum of the local 
analytic branches of $D$ passing through $x$.\\

We say that a surface singularity $x$ in $(X,D)$ is 
{\em of type $(n,q;m_1,m_2)$},
where $n,q,m_1$ and $m_2$ are positive integers, 
if the following conditions hold:
\begin{itemize}
\item $G(x)$ is a tree whose vertices have degree at most two,
\item $D_x$ consists of two curves of coefficients $1-\frac{1}{m_1}$ and $1-\frac{1}{m_2}$ 
and the vertices corresponding to $f^{-1}_* D_x$ have degree one,
\item $E_x$ has simple normal crossing singularities, 
its irreducible components are smooth rational curves 
and the sequence of weights of $E_x$ in $G(x)$
equals the sequence realizing $n/q$ as a continued fraction.
\end{itemize}

In other words, the exceptional set of $f$
is a chain of $k$ smooth rational curves with self-intersections
$-b_1, \dots, -b_k$, with $b_i\geq 2$, 
\[
\frac{n}{q}= b_1 -\frac{1}{b_2 -\frac{1}{b_3-\dots}}
\qquad 
\text{ and }
\qquad
(n,q)=1.
\]
Observe that in this case
the strict transform of $D_x$,
if non-trivial, 
will correspond to one or both ends of $G(x)$.
The above fractional expression is denoted by $[b_1,\dots, b_k]$.
Singularities of type $(n,q;m_1, m_2)$ are called {\em cyclic},
and these singularities are klt if and only if
$m_1$ and $m_2$ are finite.
We indicate by $q'$ the positive integer number such that
$n/q'=[b_k,\dots, b_1]$ holds.
Given a surface singularity $x$ in $(X,D)$ of type
$(n,q;m_1, m_2)$ we denote by
\[
\delta(x)= \frac{q}{nm_1^2} + \frac{q'}{nm_2^2} - \frac{2}{nm_1m_2}
\left( 1 + nm_1+nm_2 \right)
\]
the number that we call 
the {\em contribution} of the singularity $x$.
We omit the entry of $m_i$ whenever $m_i=1$.
Moreover, a smooth point will be considered to be a cyclic point of
type $(1,0)$.
For singularities of type $(1,0;m_1,m_2)$ we formally define
\[
E_x^2=-2  \quad \text{ and } \quad \delta(x)= -2 - \frac{4}{m_1m_2}.
\]
The above formal definitions are conventions for the equalities of Remark~\ref{change} to hold.
We say that a surface singularity $x$ in $(X,D)$ 
is of {\em type} \[(b,(n_1,q_1;m_1),(n_2,q_2;m_2),(n_3,q_3;m_3))\]
if the following conditions hold:
\begin{itemize}
\item $G(x)$ is a tree with three branches and a central vertex of weight $b\geq 2$,
each branch of $G(x)$ corresponds to the weighted graph of a singularity of 
type $(n_i,q_i;m_i)$, with $m_in_i\geq 2$, for $i\in \{1,2,3\}$,  
\item $D_x$ consists of three curves of coefficients $1-\frac{1}{m_1}, 1-\frac{1}{m_2}$
and $1-\frac{1}{m_3}$, and
\item $E_x$ has simple normal crossing singularities 
and its irreducible components are smooth rational curves.
\end{itemize}
Such singularities will be called {\em platonic}.
Given a platonic singular point $x$ in $(X,D)$ we denote by
\[
\delta(x) = \sum_{i=1}^3 \frac{q'_i}{n_im_i^2} - \sum_{i=1}^3 \frac{2}{m_i} +2
\]
the number that we call the
{\em contribution} of the singularity $x$.
Observe that for a cyclic or platonic singularity $x$ of a klt surface $(X,D)$ we have that $\delta(x)\in [-5,6]$.
\end{definition}

\begin{definition}
A {\em basket of singularity} is the data of the graph of the minimal resolution, together 
with the intersection matrix of the exceptional curves of the resolution. 
Recall that surface klt singularities are rigid, in the sense
that they are uniquely determined by the minimal resolution, up to analytic local isomorphism.
\end{definition}

The following theorem gives a characterization
of klt surface singularities in terms of
the baskets of singularities just defined. 
See for example
~\cite[Theorem 3.1 and Appendix]{Kob90}
and ~\cite[Theorem 1.6]{Meg99}.

\begin{theorem}\label{class}
Let $(X,D)$ be a $2$-dimensional klt pair,
such that $D$ has standard coefficients.
Then any point $x$ in $(X,D)$ is either
\begin{itemize}
\item cyclic with $m_1$ and $m_2$ finite, or  
\item platonic with $m_1,m_2$ and $m_3$ finite, and $\sum_{i=1}^3 \frac{1}{n_im_i} >1$.
\end{itemize} 
\end{theorem}

\begin{notation}\label{totalchern}
Given a klt pair $(X,D)$
we denote by 
\[{\rm Ch}(X,D)= 4c_2(X,D)-c_1^2(X,D)\] 
and call this number the {\em Chern value} of the pair.
Observe that for a $2$-dimensional pair such that $K_X+D$ is pseudo-effective the
Chern value is non-negative. Indeed, by Theorem~\ref{langer},
we have that 
\[{\rm Ch}(X,D) = 4c_2(X,D)-c_1^2(X,D) \geq \frac{1}{3}c_1^2(X_{\text{min}},D_{\text{min}})\geq 0,\] 
where $(X_{\text{min}},D_{\text{min}})$ is the minimal model of $(X,D)$.
Moreover, for any birational contraction
$f\colon (X,D) \rightarrow (X',D')$ of
$2$-dimensional klt pairs
we have that 
\[
{\rm Ch}(f) := {\rm Ch}(X,D)- {\rm Ch}(X',D') > 0
\] 
(see ~\cite[Theorem 4.2]{Meg99}).
This inequality, 
which in ~\cite{Meg99} is called a {\em local version}
of BMY inequality, 
can also be deduced from asymptotic Riemann-Roch 
and the positivity of the modified Euler characteristic (see ~\cite[Corollary 0.1]{Lan01}).
Observe that ~\cite[Theorem 4.2]{Meg99} states the above inequality 
for log canonical pairs and the inequality is not strict, 
however in ~\cite[Page 274]{Meg99} the author points out that the 
equality for log canonical pairs implies that $m=\infty$,
therefore the inequality is strict for klt pairs.
Finally, observe that for any member $(X,D)$ of $\mathcal{F}(R, \epsilon)$
we have that $R \geq {\rm Ch}(X,D)$ by definition.
As before, if the boundary is trivial, 
we just write $\Ch(X)$ for the Chern value of $X$.
\end{notation}

\begin{remark}
It is clear that whenever we contract a curve
in the minimal model program of a smooth surface $X$
the Chern value drops by $5$. Therefore, 
we can trivially bound the numbers of steps of the MMP
of a smooth surface of general type by $\lceil \Ch(X)/5 \rceil$.
Nevertheless, in the singular case, for a divisorial contraction
$f\colon X \rightarrow X'$ the difference $\Ch(f)$ can be arbitrarily small
if we do not impose conditions on the singularities, see Example~\ref{failsklt}.
\end{remark}

In what follows we introduce further notation
that will be used in the proof of the main theorem.

\begin{notation}\label{notation}
Let $f\colon (X,D) \rightarrow (X',D')$ be a divisorial contraction of 
$2$-dimensional klt pairs,
let $C$ be the curve being contracted by $f$,
$x_0$ the image of $C$
and $x_1,\dots, x_k$ the singular points of $X$ contained in $C$.
We denote by $\widetilde{X'}$ the minimal log resolution of $x_0$
and by $\widetilde{X}$ the minimal log resolution of the points
$x_1,\dots, x_k$.
Thus, we have a commutative diagram 
\[
 \xymatrix{
\widetilde{X}\ar[r]^-{\pi} \ar[d]^-{\widetilde{f}} & X\ar[d]^-{f} \\
\widetilde{X}'\ar[r]^-{\pi'} & X' 
 }
\]
where $\widetilde{f}$ factors into blow ups of smooth centers and
we denote by $\nu(f)$ such number of blow ups.
Observe that $\nu(f)>0$ if and only if 
the strict transform of $C$ in $\widetilde{X}$ is a $(-1)$-curve.
We often write $\widetilde{C}$ for the strict transform of $C$ in $\widetilde{X}$, $c(f)$ for the negative self-intersection of $\widetilde{C}$ 
and $m(f)$ for the positive integer number such that $C$ has coefficient $1-\frac{1}{m(f)}$ in $D$.

Moreover, we will write $\widetilde{f}= h_1 \circ \dots \circ h_{\nu(f)}$ 
for the factorization of $\widetilde{f}$ into blow ups $h_i$ of smooth centers. 
We will denote by $\widetilde{g}_i$ the birational contraction $h_1 \circ \dots  \circ h_i$,
and by $g_i$ the birational morphism obtained from $\pi \circ h_1\circ \dots \circ h_i$ by contracting all 
the curves of the codomain which are exceptional over $X'$ except the exceptional curve of $h_i$.
Therefore, we have the following relation $\widetilde{g}_{i+1}=\widetilde{g}_i \circ h_{i+1}$, 
for any $i\in \{0,\dots, \nu(f)-1\}$, where $\widetilde{g}_0={\rm id}_{\widetilde{X}'}$.
Recall from Notation~\ref{blowupgraphs} that every blow-up $h_i$, with $i\in \{1,\dots, \nu(f)\}$, 
corresponds to the blow-up of a vertex or an edge of the graph $G(X',D';x_0)$.
We shall use the above notation every time that we consider
a divisorial contraction of klt pairs.
\end{notation}

\begin{remark}\label{change}
Using Notation~\ref{notation}, for a divisorial contraction 
$f\colon (X,D)\rightarrow (X',D')$ of $2$-dimensional klt pairs 
we have the following equality 
\[
c_2(X,D) - c_2(X',D') = \frac{2-k}{m(f)} - \frac{1}{r(x_0)} + \sum_{i=1}^k \frac{1}{r(x_i)},
\]
which follows from the definition of the second Chern number
as an orbifold Euler number (see ~\cite[Theorem 4.2]{Meg99}).
Moreover, for a klt surface singularity $x$ in $(X,D)$
the following holds
\[
E_{x}^2 - \delta(x) = \frac{4}{r(x)} + \left( c_1^2(Y,f^{-1}_*D)-c_1^2(X,D)\right),
\]
where $f\colon Y \rightarrow X$ is the minimal resolution of $x$ in $(X,D)$.
Using the above equalities one can compute the value of
$\Ch(f)$ in terms of the data defined in Notation~\ref{notation},
as in the proof of ~\cite[Theorem 4.2]{Meg99}. 
In what follows we will introduce some further notation
in order write $\Ch(f)$ in a more compact way in Lemma~\ref{diff}.
\end{remark}

\begin{notation}
Consider a point $x$ in $(X,D)$ which is smooth 
and is contained in two analytic branches of $D_x$
with coefficients $1-\frac{1}{m_1}$ and $1-\frac{1}{m_2}$,
that intersect transversally at $x$.
We define
\[
\gamma(x)=-2\left( 1-\frac{1}{m_1}\right)\left( 1-\frac{1}{m_2} \right),
\]
and $\gamma(x)=0$ in any other case.
\end{notation}

\begin{notation}
Given a divisorial contraction as in ~\ref{notation}, we summarize the notation
in the following formulas:
\begin{align*}
\mu(f)  & =  \nu(f) - E_{x_0}^2 +   \sum_{i=1}^k E_{x_i}^2 ,  \\
\delta(f) & =  \delta(x_0) - \sum_{i=1}^k \delta(x_i), \\
M(f) & =  4\left( \frac{m(f)-k+1}{m(f)} \right) +c(f)\left( \frac{1-m(f)^2}{m(f)^2}\right),\\
\gamma(f) & =  \sum_{i=1}^k \gamma(x_i). 
\end{align*}
The above quantities signify the following: 
\begin{itemize}
\item $\mu(f)$ is the contribution given by $c_1(X,D)^2-c_1(X',D')^2$ to $\Ch(f)$,
\item $\delta(f)$ is the correction to $\Ch(f)$ produced by the singularities,
\item $M(f)$ is the contribution to $\Ch(f)$ given by the coefficient of $C$,
\item and $\gamma(f)$ is the correction term introduced when an irreducible curve of $D$
intersects $C$ transversally.
\end{itemize}
Recall that we have a relation $\widetilde{g}_{i+1}=\widetilde{g}_i \circ h_{i+1}$ 
for $i\in \{0,\dots, \nu(f)-1\}$, and $g_{\nu(f)}=f$. 
In Claim~\ref{nondec}, we will compute $\mu(f)$ by inductively computing 
$\mu(g_i)$ for $i\in \{1,\dots, \nu(f)\}$.
\end{notation}

\section{Proof of Boundedness}

In this section we prove the main theorem.
We start by giving some examples in which the theorem 
fails when the assumptions are weakened:

\begin{example}\label{failsklt}
If we assume the singularities of the pairs in $\mathcal{F}(R,\epsilon)$
to be $\epsilon$-plt instead of $\epsilon$-klt, then
Theorem~\ref{maintheorem} does not hold.
Indeed, let $X$ be any smooth surface of general type
and consider the set 
$\{ (X_m,D_m) \mid m\in \zz_{\geq 1}\}$
of birational models of $X$  
constructed inductively as follows:
Once we have constructed $(X_m,D_m)$ there is a birational morphism
$f_m\colon X_m \rightarrow X$  and we define
$X_{m+1}$ to be the blow up of $X_m$ at a smooth point
not contained in the exceptional locus of $f_m$
and we let $D_{m+1}$ to be the exceptional divisor 
of $\psi_m \colon X_{m+1}\rightarrow X_m$ with coefficient 
$ 1- 1/2^{m+1}$, meaning that
\[
D_{m+1}= \left( 1-\frac{1}{2^{m+1}}\right) {\rm Exc}(\psi_{m}).
\]
To start the induction we set $(X_1,D_1)=(X,0)$.
Observe that all such models are $\frac{1}{2}$-plt.
We can compute
\[
\Ch(\psi_m) = \Ch(X_{m+1},D_{m+1}) - \Ch(X_{m},D_{m}) = \frac{1}{2^{m-1}}+\frac{1}{2^{2m+2}}.
\]
Indeed, we have that
\[
\mu(\psi_m) = 1 \quad \delta(\psi_m) = -4 + \frac{1}{2^{m-1}} \quad M(\psi_m)=3+\frac{1}{2^{2m+2}} \quad  \gamma(\psi_m)=0
\]
and by Lemma~\ref{diff}, the equality 
\[
\Ch(\psi_m)=\mu(\psi_m)+\delta(\psi_m)+M(\psi_m)+\gamma(\psi_m)
\] 
holds.
So, the set
$\{ (X_m,D_m) \mid m\in \zz_{\geq 1}\}$
contains pairls of arbitrarily big Picard number
and all such models satisfy
\[
\Ch(X_n,D_n) \leq \Ch(X)+3.
\]
Thus, imposing the conditions on the discrepancies
of the divisors which are not $f$-exceptional is essential in Definition~\ref{def1}.
\end{example}

\begin{example}
For a birational contraction of klt pairs 
the quantity $-c^2_1(X,D)$ will always decrease,
however the second Chern number can increase.
Thus, imposing an upper bound for the second Chern number
is not enough to obtain boundedness of singular surfaces.
For instance, consider a surface $X'$ of general type
and blow up a smooth point $x_0$, then a point in the exceptional 
curve and finally the intersection of the two exceptional curves,
we obtain a new model $\widetilde{X}$ which is obtained from
$X'$ by blowing up two vertices and one edge
in $G(x_0)$. 
Let $X$ be the singular surface obtained by contracting
the $(-2)$-curve and the $(-3)$-curve of $\widetilde{X}$,
then we have a divisorial contraction $X\rightarrow X'$ 
that contracts the image of the $(-1)$-curve of $\widetilde{X}$
to the smooth point $x_0$ of $X'$. 
Using Remark~\ref{change} we can see that 
$c_2(X') = c_2(X) + 1/6$.
Inductively we can produce a set of $\frac{1}{3}$-klt surfaces containing
models of arbitrarily big Picard number and bounded second Chern number.
\end{example}

Now we turn to give a proof of the main theorem.
The strategy will be as follows: 
First we prove that the pairs in the set $\mathcal{F}(R,\epsilon)$
have finitely many possible baskets of singularities, 
then we use such result to prove that the MMP of any member
of this set has bounded number of steps, 
this latter fact plus the boundedness of minimal models
allows us to conclude the proof.

\begin{definition}
For a positive integer $L$ we define  
$\mathcal{S}(L, \epsilon)$ to be the set of graphs
of singularities of $2$-dimensional $\epsilon$-klt pairs,
such that any connected
subgraph all whose weights are equal to $2$ has at most $L$ vertices.
We will also write $\mathcal{S}(L, \epsilon)$ for the corresponding
set of baskets of klt surface singularities.
\end{definition}

\begin{lemma}\label{fin}
Let $\epsilon \in (0,1)$ and $L$  be a positive integer,
then the set $S(L,\epsilon)$ is finite.
\end{lemma}

\begin{proof}
Given $\epsilon\in (0,1)$, 
by \cite[Lemma 3.3]{Ale93}
or \cite[Theorem 5.2]{Ale94},
we know that 
there exists a positive integer $N(\epsilon)$
such that $\sum_{v} \left( w(v)-2 \right) \leq N(\epsilon)$
for every singularity of a $2$-dimensional $\epsilon$-klt pair,
where the sum runs over all the vertices of $G(x)$ which are weighted.
Therefore, we have a bound on the sum of the weights which
are greater or equal to $3$ and $L$ gives a bound 
on the number of vertices with weight $2$, concluding the proof.
\end{proof}

\begin{lemma}\label{diff}
Using the notation of ~\ref{notation}.
Let $f\colon (X,D) \rightarrow (X',D')$ be a birational contraction of 
$2$-dimensional klt pairs.
Then we have that
\[
\Ch(f)= \mu(f) +\delta(f)+M(f)+\gamma(f)>0.
\]
\end{lemma}

\begin{proof}
This follows from the proof of ~\cite[Theorem 4.2]{Meg99}.
\end{proof}

\begin{proof}[Proof of the main Theorem]
During the proof we will use Notation~\ref{notation}
every time that we consider a birational contraction $f$.
First, recall that by Theorem~\ref{langer} 
and Theorem~\ref{existence}
we have that the minimal model of any pair in $\mathcal{F}(R,\epsilon)$
exists, has volume $(K_{\rm min} + D_{\rm min})^2$ 
bounded by $3R$ and $\epsilon$-klt singularities.
We denote the class of these log minimal models 
by $\mathcal{F}_{\rm min}(R,\epsilon)$ and we
observe that by ~\cite[Theorem 7.7]{Ale94}
$\mathcal{F}_{\rm min}(R,\epsilon)$ forms a bounded family.
We will prove that there exists a constant $s>0$,
only depending on $R$ and $\epsilon$, 
such that for any birational contraction $f$ in the MMP
of a pair in $\mathcal{F}(R,\epsilon)$ 
the inequality $\Ch(f)\geq s$ holds.
In order to do so we start by proving the following claims.

\begin{claim}\label{bound}
There exists a positive real number $B$, 
which only depends on $R$ and $\epsilon$, 
such that $\mu(f)\leq B$
for every birational contraction $f$ 
of the MMP of a pair in $\mathcal{F}(R,\epsilon)$.
\end{claim} 

\begin{claim}\label{nondec}
Let $f$ be as in Notation~\ref{notation}.
Then the inequality  $\mu(f)\geq -3$ holds,
and $\mu(g_{i+1})\geq \mu(g_i)$ for every $i\in \{1,\dots, \nu(f)-1\}$. 
Moreover, we have that 
\[
\mu(g_{i+1})= \mu(g_i)+1
\]
whenever $h_{i+1}$ corresponds to the blow-up of a vertex of degree one and weight one.
\end{claim}

\begin{claim}\label{baskets}
There exists a positive integer $L$, only depending 
on $R$ and $\epsilon$, such that 
the singularities of any model appearing in the MMP
of a pair in $\mathcal{F}(R,\epsilon)$ belong to $\mathcal{S}(L,\epsilon)$.
\end{claim}

\begin{proof}[Proof of the Claim~\ref{bound}]
Let $f\colon (X,D)\rightarrow (X',D')$ be a birational contraction
of the MMP of a pair in $\mathcal{F}(R,\epsilon)$.
Recall that by Lemma~\ref{diff} we have that
\[
\Ch(f) = \mu(f) + \delta(f) + M(f) + \gamma(f),
\]
and 
\[ 
\Ch(f) = \Ch(X,D) - \Ch(X',D') \leq \Ch(X,D) \leq R
\] 
by the positivity of $\Ch(X',D')$ 
and definition of $\mathcal{F}(R,\epsilon)$.
Therefore, it is enough to find lower bounds for 
$\delta(f)$, $M(f)$ and $\gamma(f)$. 
Since
\[
M(f) = 4\left( \frac{m(f)-k+1}{m(f)} \right) +c(f)\left( \frac{1-m(f)^2}{m(f)^2}\right),
\]
we will proceed by providing bounds for $c(f)$ and $k$:
\begin{itemize}
\item Bounding $c(f)$: 
If $\nu(f)>0$ we observe that the strict transform of $C$ in $\widetilde{X}$
is the unique $(-1)$-curve in the exceptional locus of $\widetilde{X}\rightarrow \widetilde{X'}$,
so $c(f)=1$. 
Otherwise $c(f)$ is the weight of a vertex of the graph of $x_0$, 
and therefore it does not exceed $\frac{2}{\epsilon}$ (see ~\cite[Lemma 5.20]{Ale94}).
We conclude that in any case $1 \leq c(f)\leq \frac{2}{\epsilon}$.
\item Bounding $k$: 
Observe that $k\leq 3$. 
Indeed, the vertex corresponding to $\widetilde{C}$ 
in the graph obtained by blowing up the graph of $x_0$ 
has degree at most $3$ by the classification of klt singularities ~\ref{class}.
If $k\leq 1$, then we can increase $k$ by picking smooth points in $C$.
So, without loss of generality we can assume $k \in \{2,3\}$.
\end{itemize}
Since $m(f)\in \mathbb{Z}_{\geq 1}$ we conclude that
\[
M(f) \geq  -4 - \frac{2}{\epsilon}.
\]
Finally, observe that by the definitions of $\delta(x)$ 
and $\gamma(x)$ for a single singularity $x$ of $(X,D)$ we have the following inequalities
\[
5\geq \delta(x) \geq -6 
\quad
\text{ and }
\quad
0\geq \gamma(x)\geq -2,
\]
and since $k\in \{2,3\}$ 
then we obtain bounds for $\delta(f)$ and $\gamma(f)$. 
Putting all these bounds together we conclude that
it is enough to take 
\[
B=R+31+\frac{2}{\epsilon} \geq \mu(f).
\]
\end{proof}

\begin{proof}[Proof of the Claim~\ref{nondec}]
First we analyze the value of 
\[
E_{x_1}^2+E_{x_2}^2-E_{x_0}^2 
\] 
after the first smooth
blow up $h_1$ in the graph $G(x)$ of a klt singularity $x\in \widetilde{X}'$.
We proceed in two cases depending whether the first blow up $h_1$
is at a vertex or an edge: 
\begin{itemize}
\item We blow up an edge $e$ joining vertices $v_1$ and $v_2$
with $w(v_1),w(v_2)\geq 2$:\\
In this case $w(v_1)$ and $w(v_2)$ increase by one and
we can compute 
\[ \mu(g_1) =  (E_{x_1}^2+E_{x_2}^2-E_{x_0}^2) + 1 = -3. \]
\item We blow up a vertex $v$ with weight $w(v)\geq 2$:\\
In this case we increase the weight of $v$ by one 
and introduce a new vertex of weight one which corresponds
to the exceptional curve of the blow up.
We let $x_2$ be any smooth point in such curve
so we obtain
\[ \mu(g_1) = (E_{x_1}^2+E_{x_2}^2-E_{x_0}^2) + 1 =2. \]
\end{itemize}
Now we turn to prove that after the first blow up $h_1$,
the number $\mu(g_i)$ can not decrease, when $i$ increases in $\{1,\dots,\nu(f) \}$.
Recall that since $\widetilde{X}$ contains
at most one $(-1)$-curve which is exceptional over $\widetilde{X'}$
then any further blow up $h_i$, with $i\in \{2, \dots, \nu(f)\}$, 
is performed at a vertex of weight one or at one of its edges.
Moreover, since the vertex of weight one introduced in the first blow up
has degree at most two, we only have to analyze four different cases.
Again, we proceed case by case:
\begin{itemize}
\item We blow up a vertex of weight one and degree one:\\
In this case $-E_{x_0}^2, E_{x_1}^2$ and $E_{x_2}^2$ remain
the same after the blow up $h_{i+1}$, so that $\mu(g_{i+1})=\mu(g_i)+1$.
Indeed, the singularities of $x_0$ and $x_2$ do not change,
while the graph of $x_1$ is enlarged with an edge and a vertex of weight $2$.
\item We blow up a vertex of weight one and degree two:\\
In this case the values $-E_{x_0}^2$ 
and $E_{x_1}^2+E_{x_2}^2$ remain the same
after the blow up $h_{i+1}$, so that $\mu(g_{i+1})=\mu(g_i)+1$. 
Indeed, the graph of $x_1$ is enlarged with 
two edges, a vertex of weight $2$ and the graph of $x_2$, 
while the new point $x_2$ corresponds to a smooth point in
the exceptional curve.
\item We blow up an edge of a vertex of weight one and degree one:\\
In this case, after the blow-up $h_{i+1}$, the singularity $x_0$ does not change, $E_{x_1}^2$ decreases by one,
$E_{x_2}^2$ remains the same, so that $\mu(g_{i+1})=\mu(g_i)$.
\item We blow up an edge of a vertex of weight one and degree two:\\ 
In this case, after the blow-up $h_{i+1}$, the singularity $x_0$ does not change, and up to permuting $x_1$ and $x_2$
we can assume that $E_{x_1}^2$ remains the same and $E_{x_2}^2$ decreases by one.
Thus, we have that $\mu(g_{i+1})=\mu(g_i)$.
\end{itemize}
Finally, observe that after any such blow up,
the vertex of weight one corresponding to the exceptional curve
will have degree at most two as well, 
so these four cases are the only ones that can happen 
in the sequence of blow ups that factors $\widetilde{f}$.
\end{proof}

\begin{proof}[Proof of the Claim~\ref{baskets}]
Observe that since $\mathcal{F}_{\rm min}(R,\epsilon)$
is a bounded family then there exists $L_0$
such that the baskets of singularities appearing in this family
belong to $\mathcal{S}(L_0,\epsilon)$. 
Indeed, it is enough to take $L_0$ to be the maximum 
length of the subgraph of rational $(-2)$-curves  
in the minimal resolution of the singularities which appear on $\mathcal{F}_{\rm min}(R,\epsilon)$.
Now we turn to prove that it is enough to take
\[
L = \max \left\{ L_0 , R+36+\frac{2}{\epsilon} \right\}. 
\]
Let $f$ be a contraction in the minimal model program
of a pair in $\mathcal{F}(V,\epsilon)$, 
we will prove that if the singularity $x_0$ belongs
to $\mathcal{S}(L, \epsilon)$, then the singularities
$x_1,\dots, x_k$ belong to $\mathcal{S}(L,\epsilon)$ as well.
If $\nu(f)=0$ then the statement holds, since
the graph of any point $x_1,\dots,x_k$ is a subgraph of 
the graph of $x_0$.
Moreover, if a chain of $(-2)$-curves in the graph of $x_1$ is already a chain
of $(-2)$-curves in the graph of $x_0$ then the length 
of such chain is at most $L$ by definition, 
so it is enough to consider chains of $(-2)$-curves in the 
graph of $x_1$ which are obtained by $\widetilde{f}$.

From now, we assume that $\nu(f)>0$.
Observe that in this case, 
after the first blow up $h_1$ of the graph of $x_0$,
the $(-1)$-curve that we introduce intersects curves of
negative self-intersection at least three, so we conclude that 
the connected chains of $(-2)$-curves that we introduce by $\widetilde{f}$
can not intersect the chains of $(-2)$-curves in the graph of $x_0$. 
Here, we are using the fact that for platonic singularities, the central component of the fork
has self-intersection $-b$, with $b\geq 2$.
Now, we analyze what happens to the chains of $(-2)$-curves
in the blow up of the graph of $x_0$ when we blow up
an edge after a vertex, or vice versa, in the factorization of $\widetilde{f}$:
\begin{itemize}
\item The factorization of $\widetilde{f}$ into blow ups of smooth centers contains a blow 
up of an edge followed by a blow up of a vertex:\\
In this case, we blow up an edge $e$ joining two vertices $v_1$ and $v_2$, 
and we introduce a vertex $v_0$ of weight one. 
Since $\widetilde{X}$ contains only one $(-1)$-curve exceptional over $\widetilde{X'}$
we conclude that the next blow up is at $v_0$, 
increasing the weight of $v_0$ to two and introducing a vertex $v$.
In this case we have three possible sub-cases:
\begin{itemize}
\item If there are no further blow ups, 
we conclude that any chain of $(-2)$-curves obtained before 
the last blow up can increase at most by one vertex due to $v_0$. 
\item If the next blow up is at an edge, 
it has to be at the edge joining $v$ and $v_0$, 
in this case the weight of $v_0$ increases to three
and then any $(-2)$-curve that we introduce after this last blow up
will be disjoint from the previous chains of $(-2)$-curves.
\item If the next blow up is at a vertex 
it has to be at $v$.
In this case we can assume without loss of generality that
$w(v_1)\geq 3$, since before the blow up at $e$
one of the vertices $v_1$ and $v_2$ have weight at least two. 
Observe that in this case, after the blow up at $v$, 
the vertex $v_0$ has degree three and then using the classification
of platonic singularities~\ref{class} we conclude that the only possibly
unbounded branch is the one at $v_1$.
Moreover, the complement of the branch at $v_1$ in this graph
has at most seven vertices.
Thus, we conclude that any chain of $(-2)$-curves of length 
more than seven was already contained in the branch of $v_1$ before the blow up at $e$. 
\end{itemize}
\item The factorization of $\widetilde{f}$ into blow ups of smooth centers contains a blow
up of a vertex followed by a blow up of an edge:\\ 
Since $\widetilde{X}$ contains at most one $(-1)$-curve which is exceptional over
$\widetilde{X'}$ we see that after blowing up a vertex $v$ of $G(x)$ 
the only edge that we can blow up is the edge introduced by the preceding blow up.
In this case, after blowing up the edge we increase the weight of the 
vertex $v$ by at least two, so we discontinue any chain of $(-2)$-curves previously
introduced.
\end{itemize}
Thereof, we conclude that whenever we shift from blowing up 
vertices to edges, or vice versa, we end any previous chain
of $(-2)$-curves after possibly adding a last $(-2)$-curve.
Finally, we argue that the chains of $(-2)$-curves produced by blowing up
a sequence of edges or a sequence of vertices have length bounded by $L-1$:
\begin{itemize}
\item The chain of $(-2)$-curves is obtained by blowing up a sequence of edges:\\
Assume that after a sequence of blow ups of edges we obtain a chain of 
$(-2)$-curves of length $l$. 
Since the last blow up of this sequence is at an edge,
then the last exceptional curve correspond to a vertex $v_0$ which is
joined to the chain of $(-2)$-curves and another vertex $v_1$. 
By succesively blowing down the $(-1)$-curve,
at each step we decrease the weight of $v_1$ 
and the length $l$ by one.
Recall that every graph obtained by blowing up graphs of klt singularities
has positive weights, so we conclude that $w(v_1)>l$.
Moreover, we know that all the weights of graphs of $\epsilon$-klt singularities
are $\leq \frac{2}{\epsilon}$ (see ~\cite[Lemma 5.20]{Ale94}) 
concluding that $\frac{2}{\epsilon}>l$.
\item The chain of $(-2)$-curves is obtained by blowing up a sequence of vertices:\\
Assume that after a sequence of blow ups of vertices we obtain 
a chain of $(-2)$-curves of length $l$. 
Recall that after blowing up a vertex $v_0$
we introduce a new vertex of weight and degree one, 
then we deduce that at least $l-1$ of the blow ups were 
at vertices of degree one and weight one.
Thus, by Claim~\ref{nondec} we conclude that $\mu(f)+3\geq l-1$
and by Claim~\ref{bound} we see that  $B+4 \geq l$.
\end{itemize}
By the last two cases we conclude that any chain of $(-2)$-curves 
obtained by blowing up a sequence of edges or a sequence of vertices has length at most $L-1$,
and using the first two cases we see that any chain has length at most $L$.
Thereof we conclude that the graph of $x_1$ contains chains of $(-2)$-curves
of length at most $L$ as claimed. 
\end{proof}

Now we return to prove that such $s>0$ exists: 
By Claim~\ref{baskets}, Lemma~\ref{fin}, and the bound $k\leq 3$, the set
\[
\{ \delta(f) \mid \text{ $f$ is a divisorial contraction of the MMP of a pair in $\mathcal{F}(R,\epsilon)$}\},
\]
is finite.
Moreover, by Theorem~\ref{existence}
any pair which appears in the minimal model program of
an element in $\mathcal{F}(R,\epsilon)$
has $\epsilon$-klt singularities, 
so that $\frac{1}{\epsilon} \geq m(f) \geq 0$
and $\frac{2}{\epsilon} \geq c(f) \geq 1$.
Thus, we conclude that the sets
\[
\{ M(f) \mid \text{ $f$ is a divisorial contraction of the MMP of a pair in $\mathcal{F}(R,\epsilon)$}\}
\]
and 
\[
\{ \gamma(f) \mid \text{ $f$ is a divisorial contraction of the MMP of a pair in $\mathcal{F}(R,\epsilon)$}\}
\]
are finite as well.
Since $\mu(f)$ is an integer number, 
we conclude by Lemma~\ref{diff} that $\Ch(f)$ belongs to $\zz[\frac{1}{N}]$ 
for some integer number $N$, depending only on $R$ and $\epsilon$.
Finally, considering that $\Ch(f)$ is strictly positive we conclude that
there exists $s>0$, only depending on $R$ and $\epsilon$, 
such that $\Ch(f)\geq s$.\\

Now, we can finish the proof of the theorem.
The MMP of any pair in $\mathcal{F}(R,\epsilon)$
has at most $\lceil R/s \rceil$ steps,
and in any step we introduce at most three singular points,
all of them contained in the finite set of baskets $\mathcal{S}(L,\epsilon)$.
Using ~\cite[Lemma 3.8]{Ale94} we see that
the set $\mathcal{F}_{\rm sm}(R,\epsilon)$ 
of minimal log resolutions of models in $\mathcal{F}(R,\epsilon)$ forms a bounded
family. 
Therefore, we can choose a bounding $\qq$-polarization $\mathcal{A}_0$ for the class $\mathcal{F}_{\rm sm}(R,\epsilon)$,
in the sense of~\ref{qpolarization}. Then, we have an induced $\qq$-polarization $\mathcal{A}$ induced
by pushing-forward the $\qq$-Cartier $\qq$-divisors of $\mathcal{A}_0$ to the log pairs $(X,D)\in \mathcal{F}(R,\epsilon)$.
Since the $\qq$-Cartier $\qq$-divisors of $\mathcal{A}_0$ have bounded Cartier index
and the log pairs of $\mathcal{F}(R,\epsilon)$ have $\epsilon$-klt singularities, we conclude that
the $\qq$-Cartier $\qq$-divisors of $\mathcal{A}$ have bounded Cartier index as well.
Finally, we need to check the second and third conditions of~\ref{qpolarization} for the $\qq$-polarization $\mathcal{A}$.

Let $(X,D) \in \mathcal{F}(R,\epsilon)$ and $A_X$ the corresponding ample $\qq$-divisor on $X$, 
then the minimal resolution $\pi \colon (X_0, D_0+E_0)\rightarrow (X,D)$ of $(X,D)$ belongs to $\mathcal{F}_{\rm sm}(R,\epsilon)$,
where $E_0$ is the $\pi$-exceptional effective divisor such that $\pi^*(K_X+D)=K_{X_0}+E_0+D_0$.
Then, by the negativity lemma we can write $\pi^*(A_X)=A_{X_0}+E_{X_0}$, 
where $A_{X_0}$ is the ample $\qq$-divisor on $X_0$ induced by the bounding $\qq$-polarization $\mathcal{A}_0$
and $E_{X_0}$ is an effective $\pi$-exceptional divisor whose intersection matrix is negative definite.
By the projection formula and the fact that $\mathcal{A}_0$ is a bounding $\qq$-polarization, we have that the value
\[
A_X^2 = (A_{X_0}+E'_{X_0})^2 \leq A_{X_0}^2 + A_{X_0}\cdot E_{X_0} \leq A_{X_0}^2+ A_{X_0}\cdot (E_{X_0})_{\rm red},
\]
is bounded independent of $(X,D)\in \mathcal{F}(R,\epsilon)$.
Analogously, we have that the value
\[
A_X\cdot D = (A_{X_0}+E_{X_0})\cdot D_0 \leq A_{X_0}\cdot D_0 + (E_{X_0})_{\rm red}\cdot (D_0)_{\rm red},
\]
is bounded independent of $(X,D)\in \mathcal{F}(R,\epsilon)$.

\end{proof}


\begin{bibdiv}
\begin{biblist}

\bib{Flips}{collection}{
   title={Flips and abundance for algebraic threefolds},
   note={Papers from the Second Summer Seminar on Algebraic Geometry held at
   the University of Utah, Salt Lake City, Utah, August 1991;
   Ast\'erisque No. 211 (1992) (1992)},
   publisher={Soci\'et\'e Math\'ematique de France, Paris},
   date={1992},
   pages={1--258},
   issn={0303-1179},
   review={\MR{1225842}},
}
	

\bib{Ale93}{article}{
   author={Alexeev, Valery},
   title={Two two-dimensional terminations},
   journal={Duke Math. J.},
   volume={69},
   date={1993},
   number={3},
   pages={527--545},
   issn={0012-7094},
   review={\MR{1208810}},
   doi={10.1215/S0012-7094-93-06922-0},
}

\bib{Ale94}{article}{
   author={Alexeev, Valery},
   title={Boundedness and $K^2$ for log surfaces},
   journal={Internat. J. Math.},
   volume={5},
   date={1994},
   number={6},
   pages={779--810},
   issn={0129-167X},
   review={\MR{1298994}},
   doi={10.1142/S0129167X94000395},
}
	

\bib{AM00}{article}{
   author={Alexeev, Valery},
   author={Mori, Shigefumi},
   title={Bounding singular surfaces of general type},
   conference={
      title={Algebra, arithmetic and geometry with applications},
      address={West Lafayette, IN},
      date={2000},
   },
   book={
      publisher={Springer, Berlin},
   },
   date={2004},
   pages={143--174},
   review={\MR{2037085}},
}

\bib{BCHM}{article}{
   author={Birkar, Caucher},
   author={Cascini, Paolo},
   author={Hacon, Christopher D.},
   author={McKernan, James},
   title={Existence of minimal models for varieties of log general type},
   journal={J. Amer. Math. Soc.},
   volume={23},
   date={2010},
   number={2},
   pages={405--468},
   issn={0894-0347},
   review={\MR{2601039}},
   doi={10.1090/S0894-0347-09-00649-3},
}

\bib{Bog78}{article}{
   author={Bogomolov, F. A.},
   title={Holomorphic tensors and vector bundles on projective manifolds},
   language={Russian},
   journal={Izv. Akad. Nauk SSSR Ser. Mat.},
   volume={42},
   date={1978},
   number={6},
   pages={1227--1287, 1439},
   issn={0373-2436},
   review={\MR{522939}},
}

\bib{Har77}{book}{
   author={Hartshorne, Robin},
   title={Algebraic geometry},
   note={Graduate Texts in Mathematics, No. 52},
   publisher={Springer-Verlag, New York-Heidelberg},
   date={1977},
   pages={xvi+496},
   isbn={0-387-90244-9},
   review={\MR{0463157}},
}

\bib{HMX14}{article}{
 author={Hacon, Christopher},
 author={McKernan, James},
 author={Xu, Chenyang},
  title={Boundedness of moduli of varieties of general type},
 JOURNAL = {https://arxiv.org/abs/1412.1186 }
      YEAR = {2014},
}

\bib{KM98}{book}{
   author={Koll\'ar, J\'anos},
   author={Mori, Shigefumi},
   title={Birational geometry of algebraic varieties},
   series={Cambridge Tracts in Mathematics},
   volume={134},
   note={With the collaboration of C. H. Clemens and A. Corti;
   Translated from the 1998 Japanese original},
   publisher={Cambridge University Press, Cambridge},
   date={1998},
   pages={viii+254},
   isbn={0-521-63277-3},
   review={\MR{1658959}},
   doi={10.1017/CBO9780511662560},
}

\bib{Koll13}{book}{
   author={Koll\'ar, J\'anos},
   title={Singularities of the minimal model program},
   series={Cambridge Tracts in Mathematics},
   volume={200},
   note={With a collaboration of S\'andor Kov\'acs},
   publisher={Cambridge University Press, Cambridge},
   date={2013},
   pages={x+370},
   isbn={978-1-107-03534-8},
   review={\MR{3057950}},
   doi={10.1017/CBO9781139547895},
}

\bib{Kob90}{article}{
   author={Kobayashi, Ryoichi},
   title={Uniformization of complex surfaces},
   conference={
      title={K\"ahler metric and moduli spaces},
   },
   book={
      series={Adv. Stud. Pure Math.},
      volume={18},
      publisher={Academic Press, Boston, MA},
   },
   date={1990},
   pages={313--394},
   review={\MR{1145252}},
}
	

\bib{Lan00}{article}{
   author={Langer, Adrian},
   title={Chern classes of reflexive sheaves on normal surfaces},
   journal={Math. Z.},
   volume={235},
   date={2000},
   number={3},
   pages={591--614},
   issn={0025-5874},
   review={\MR{1800214}},
   doi={10.1007/s002090000149},
}
	

\bib{Lan01}{article}{
   author={Langer, Adrian},
   title={The Bogomolov-Miyaoka-Yau inequality for log canonical surfaces},
   journal={J. London Math. Soc. (2)},
   volume={64},
   date={2001},
   number={2},
   pages={327--343},
   issn={0024-6107},
   review={\MR{1853454}},
   doi={10.1112/S0024610701002320},
}

\bib{Lan16}{article}{
   author={Langer, Adrian},
   title={The Bogomolov-Miyaoka-Yau inequality for logarithmic surfaces in
   positive characteristic},
   journal={Duke Math. J.},
   volume={165},
   date={2016},
   number={14},
   pages={2737--2769},
   issn={0012-7094},
   review={\MR{3551772}},
   doi={10.1215/00127094-3627203},
}

\bib{Meg99}{article}{
   author={Megyesi, G.},
   title={Generalisation of the Bogomolov-Miyaoka-Yau inequality to singular
   surfaces},
   journal={Proc. London Math. Soc. (3)},
   volume={78},
   date={1999},
   number={2},
   pages={241--282},
   issn={0024-6115},
   review={\MR{1665244}},
   doi={10.1112/S0024611599001719},
}

\bib{Miy77}{article}{
   author={Miyaoka, Yoichi},
   title={On the Chern numbers of surfaces of general type},
   journal={Invent. Math.},
   volume={42},
   date={1977},
   pages={225--237},
   issn={0020-9910},
   review={\MR{0460343}},
   doi={10.1007/BF01389789},
}

\bib{Miy84}{article}{
   author={Miyaoka, Yoichi},
   title={The maximal number of quotient singularities on surfaces with
   given numerical invariants},
   journal={Math. Ann.},
   volume={268},
   date={1984},
   number={2},
   pages={159--171},
   issn={0025-5831},
   review={\MR{744605}},
   doi={10.1007/BF01456083},
}

\bib{Sak80}{article}{
   author={Sakai, Fumio},
   title={Semistable curves on algebraic surfaces and logarithmic
   pluricanonical maps},
   journal={Math. Ann.},
   volume={254},
   date={1980},
   number={2},
   pages={89--120},
   issn={0025-5831},
   review={\MR{597076}},
   doi={10.1007/BF01467073},
}

\bib{Tsu83}{article}{
   author={Tsunoda, Shuichiro},
   title={Structure of open algebraic surfaces. I},
   journal={J. Math. Kyoto Univ.},
   volume={23},
   date={1983},
   number={1},
   pages={95--125},
   issn={0023-608X},
   review={\MR{692732}},
}
	

\bib{Wahl93}{article}{
   author={Wahl, Jonathan},
   title={Second Chern class and Riemann-Roch for vector bundles on
   resolutions of surface singularities},
   journal={Math. Ann.},
   volume={295},
   date={1993},
   number={1},
   pages={81--110},
   issn={0025-5831},
   review={\MR{1198843}},
   doi={10.1007/BF01444878},
}


\bib{Yau77}{article}{
   author={Yau, Shing Tung},
   title={Calabi's conjecture and some new results in algebraic geometry},
   journal={Proc. Nat. Acad. Sci. U.S.A.},
   volume={74},
   date={1977},
   number={5},
   pages={1798--1799},
   issn={0027-8424},
   review={\MR{0451180}},
}


\end{biblist}
\end{bibdiv}
\end{document}